\documentclass[10pt, reqno]{amsart}
\usepackage{amsmath}
\usepackage{amsfonts}
\usepackage{amssymb}
\usepackage{amsthm}
\usepackage{hyperref}

\newtheorem{theorem}{Theorem}
\newtheorem{lemma}[theorem]{Lemma}

\newtheorem{proposition}[theorem]{Proposition}
\newtheorem{definition}[theorem]{Definition}

\newcommand{\spam}{\mathop{\mathrm{span}}}
\newcommand{\supp}[1]{{\text{supp}}({#1})}
\newcommand{\dist}{\mathrm{dist}}
\newcommand{\ints}{\mathbb{Z}}
\newcommand{\Z}{\mathbb{Z}}
\newcommand{\Zd}{\mathbb{Z}^d}

\newcommand{\reals}{\mathbb{R}}
\newcommand{\Rd}{\mathbb{R}^d}
\newcommand{\comps}{\mathbb{C}}

\newcommand{\dif}{\, \mathrm{d}}

\newcommand{\ful}[1]{T_h^{\sharp} #1}

\newcommand{\notful}[1]{T_{h}^{\flat} #1}
\newcommand{\final}[2]{T_{#1} #2}
\newcommand{\gapp}[1]{T_{#1}}

\newcommand{\D}{\mathcal{D}}
\newcommand{\E}{\mathcal{E}}

\newcommand{\alp}{\alpha}
\newcommand{\bet}{\beta}
\newcommand{\tet}{\theta}

\newcommand{\hatf}{\widehat{f}}
\newcommand{\hatphi}{\widehat{\phi}}
\newcommand{\norm}[1]{\|#1\|}
\newcommand{\bks}{\backslash}

\newcommand{\cd}{\mathrm{c}}

\begin{document}
\title{Nonlinear Approximation Using Gaussian Kernels}

\author{Thomas Hangelbroek}
\address{Thomas Hangelbroek, Department of Mathematics, Texas A\&M University, College Station, TX 77843}
\email{hangelbr@math.tamu.edu}
\thanks{Thomas Hangelbroek is supported by an NSF Postdoctoral Fellowship}

\author{Amos Ron}
\address{Amos Ron, Computer Science Department, 
University of Wisconsin, Madison, 
WI 53706}
\email{amos@cs.wisc.edu}
\thanks{Amos Ron is supported by the National Science Foundation under 
grants DMS-0602837 and DMS-0914986, and by the National Institute of General 
Medical Sciences under Grant NIH-1-R01-GM072000-01.}



\begin{abstract}
It is well-known that non-linear approximation has an advantage over linear
schemes in the sense that it provides comparable approximation rates
to those of the linear schemes, but to a larger class of approximands. This was
established for spline approximations and for wavelet approximations, and
more recently by DeVore and Ron \cite{DR} for homogeneous radial basis function (surface 
spline) approximations.
However, no such results are known for the Gaussian function, the preferred 
kernel in   machine learning and several engineering problems.
We introduce and analyze in this paper a new algorithm for approximating
functions using translates of Gaussian functions with varying tension
parameters. At heart it employs the strategy for nonlinear approximation of DeVore\! --\! Ron, 
but it selects kernels 
by a method that is not straightforward. 
 The crux of the
difficulty lies in the necessity to vary the tension parameter in the
Gaussian function spatially according to local information about the
approximand: error analysis of Gaussian approximation schemes  with varying
tension are, by and large, an elusive target for approximators.
We show that our algorithm is suitably optimal in the sense that it provides approximation
rates similar to  other established nonlinear methodologies like spline and 
wavelet approximations.  As expected and desired, the approximation rates 
can be as high as needed and are essentially saturated only by the smoothness
of the approximand.
\end{abstract}
\maketitle
\section{Introduction}
\subsection{Nonlinear Radial Basis Function Approximation}
In this article we consider $N$-term approximation 
by Gaussian networks, an approximation technique 
widely used in statistics and engineering.
This is an example of nonlinear approximation since we select
$d$-variate functions residing in 
$$\mathbb{G}_N:=
\left\{\sum_{j=1}^N A_j\  \exp\left(-\left|\frac{\cdot - c_j}{\sigma_j}\right|^2\right): 
A\in \comps^N, \sigma\in (0,\infty)^N, c\in \reals^{dN}\right\}$$
which (failing to be closed under addition) is not a linear space. 
This stands in contrast to the linear approximation problem, often studied
in radial basis function (RBF) theory, 
where the {\em centers} $(c_j)_j$,   
are predetermined and approximants are chosen from a linear space  
$$\spam_{1\le j\le N} \phi_j(\cdot - c_j)  = \spam_{1\le j\le N} \exp\left(-\left|\frac{\cdot - c_j}{\sigma_j}\right|^2\right)$$
that depends on the set of centers.

Heuristically, the benefit of the nonlinear approach is that  by  placing
centers strategically, one may overcome defects, like discontinuities, cusps
or other local deficiencies in smoothness, of the target function $f$. 
Because such defects may be manifested in a variety of ways, over regions
or on lower dimensional manifolds, and may occur at different scales, 
finding a precise strategy is not at all straightforward. In this article, 
we present a 
method for placing centers in a way that is 
suitable for creating effective nonlinear approximants.

An important distinction between the nonlinear and linear problems is in how 
convergence is measured. In the linear setting, the main approximation parameter measures density of the centers, usually by means of the ``fill distance'' $h = \max_{x\in \Omega} \dist\bigl(x,(c_j)_j\bigr)$; the underlying approximation problem is to measure the rate of convergence as $h$ shrinks.
In high dimensions, the assumption that centers fill
a (high dimensional) region $\Omega$ with a small fill distance is computationally impractical.  
In nonlinear approximation the rate of convergence is measured against the parameter $N$, the cardinality of the set of centers. This approach lends itself to more frugal approximation in high dimensions.

The approximation scheme we introduce selects $s_{f,N}$ from $\mathbb{G}_N$, 
and is shown to have convergence rate
$ \|f -s_{f,N}\|_p= \mathcal{O}(N^{-s/d})$ for target functions $f$ having $L_{\tau}$ smoothness $s$, with $\frac{1}{\tau} = \frac{s}{d} +\frac{1}{p}$. 
Generally speaking, such nonlinear estimates are sharp in the sense that
they are similar to known results for nonlinear wavelet approximation,
and one cannot expect to achieve a similar rate $\dist_p(f, \mathbb{G}_N) =
\mathcal{O}(N^{-s/d})$ by decreasing either $\tau$ or  the underlying
smoothness.
%

To provide a more robust space of approximants, we permit the tension (aka shape or dilation) parameters $\sigma_j$ to respond to the nonuniform distribution of the centers.
The question of how to tune a tension parameter
is of active interest to the Learning Theory community, \cite{StSc,YiZh}, 
as well as the RBF community \cite{LiMi,BR}, 
but in most theoretical works, the tension parameter is taken to be constant for all centers.
Although the spatially varying tension parameter is a natural idea, 
and is used in practice \cite{KaCa,PogVap}, 
it has heretofore not been considered seriously in an approximation theoretic sense. 
Although it may be tempting to use tight dilations when the centers are dense, essentially
setting $\sigma_j$ proportional to a local spacing of centers around $c_j$, 
the manner in which our scheme sets the tension is more complicated, 
but one that is ultimately justified by the error estimates we provide.
In any case, we note that there is some empirical evidence \cite[Section 3]{FZ}
that Gaussian approximation is unstable without adjusting the tension.
%

Nonlinear approximation with RBFs has not been investigated with the same intensity as other basic elements of approximation theory (splines, wavelets, etc.).
Recently DeVore and Ron \cite{DR} (employing a idea on which we have modeled our method) have made a first foray into nonlinear RBF approximation using RBFs that are fundamental solutions of elementary, homogeneous, elliptic PDEs. 
Such RBFs, which include the ``surface splines,'' allow simple but elegant approximation schemes that are not burdened by the requirement that the target function must reside in the native space. In addition, the homogeneity of these RBFs means that the $N$-term approximation spaces 
are, essentially, invariant under rescaling and, thus, there is no need to select dilations $\sigma_j$ -- this is done automatically.
However, many prominent RBFs, including the Gaussians, do not fall into this category. For the 
kernels considered by DeVore and Ron, 
the approximation order is {\em saturated}, meaning that for this method there is an upper bound on the rate of convergence: 
by increasing smoothness beyond a saturation level
$k$ (determined by the order of the elliptic differential operator inverted by the kernel) 
there is no corresponding increase in the rate of decay of the error.
This is not so with Gaussian kernels.
Furthermore, the kernels used by DeVore and Ron are dependent on the operator they invert, and, hence, (subtly) dependent on the spatial dimension. This is a  hindrance which the Gaussians also avoid.

\subsection{The Methodology}
As in \cite{DR}, to construct the $N$-term approximant $s_N$, 
we begin with a wavelet decomposition of the
target function $f= \sum_I f_I \psi_I$.  Based on the size of the wavelet coefficient and 
the smoothness norm of the target function, the fixed budget of $N$ terms is distributed
over the elements in the expansion -- into individual budgets $N_I$ (many of 
which are zero). Each wavelet $\psi_I$ is then approximated by a linear 
combination $s_I$ of Gaussians that uses at most $N_I$ terms. The full $N$ 
term approximant is then $s_{f,N} = \sum f_I s_I$.  The main idea is that we 
have a scheme for nonlinear approximation associated with this family of 
wavelets that can be lifted to the Gaussians by means of approximating the 
individual members of the family.  Matters are simplified when we assume the 
entire family to be generated from a few prototypes via dilation and 
translation: our collection of Gaussians are invariant under these operations!
This reduces the problem of efficiently 
approximating all members of the wavelet family to the problem of 
approximating a few fixed wavelets by linear combinations of Gaussians.

The crucial issue is to approximate a basic function $\psi$ using a linear 
combination of $N$ shifted Gaussians.  We view the number $N$ as the 
portion we are willing to invest in approximating $\psi$ out of our total
budget of centers. It is 
essential to understand how to apportion the budget, and this can only be 
accomplished when we have good $N$-term error estimates.  Thus, we are 
interested in understanding how to approximate globally using only finitely 
many centers.  This is a very hard problem for the Gaussian. We completely
resolve this problem for a function
$\psi$ that is band-limited, and in addition, has rapid decay:
$$\text{for every $k$ there is a constant $C_k$ such that } |\psi(x)| \le C_k (1+|x|)^{-k}.$$
The trick we employ is to create an approximant $\sum_{\alpha\in h\ints^d} 
a(\alpha,h) \phi(\cdot - \alpha)$ that converges rapidly (globally) to 
$\psi$ in the $L_{\infty}$ norm, with coefficients $a(\alpha, h)$ that are 
roughly the same size as $\psi(\alpha)$. Then we modify this approximation 
scheme by throwing away centers from a region where $\psi$ is small. This is 
where the two assumptions on the wavelet -- that $\psi$ is bandlimited and 
that it is rapidly decaying -- come into play. Bandlimiting means that the 
``full'' approximation scheme (using centers $h\ints^d$) has coefficients 
$a(\alpha,h)$ that can be expressed as the convolution of $\psi$ with a 
Schwartz function. Rapid decay allows us to attribute polynomial decay of 
arbitrary orders to the coefficients.

\subsection{Organization}
In Section 2 of this article, we develop the basic linear approximation 
scheme at the heart of our approach.  First considered is the operator 
$T_h^{\sharp}$, which generates the `full' approximant, an infinite series 
of Gaussians having the grid $h\ints^d$ as the set of centers.  
Second we develop the operator $T_{h}^{\flat}$, which generates the `truncated' approximant -- a linear combination of roughly $h^{-2d}$ Gaussians . 
At the end of Section 2 we generalize $T_{h}^{\flat}$ to treat scaled wavelets using a fixed budget of $N$ centers. This is the role of the map $T_{N}$.  Corollary \ref{loc_err3} gives the error for wavelets at all dilation levels.

Section 3 treats nonlinear approximation in $L_p$ for $1\le p <\infty$. Results match  those obtained for surface splines in \cite{DR}. This involves a sophisticated strategy for  distributing centers, which is expressed in Section 3.1. The main result is Theorem \ref{p_main} in Section 3.3.

Section 4 treats nonlinear approximation in $L_\infty$, a case was not
considered in \cite{DR}. For technical reasons, we consider approximation of 
functions from Besov spaces in this section. The main result in that section 
is Theorem \ref{inf_first}.

\subsection{Notation and Background}
We denote the ball with center $c$ and radius $R$ by $B(c,R).$  The symbol 
$I\subset \reals^d$ will represent a cube with corner at  $c(I)\in \reals^d$ 
and sidelength $\ell(I)>0$: it is the set $c(I) +[0,\ell(I)]^d.$  We denote the volume of a set $\Omega$ in $\reals^d$ by $|\Omega|.$

The natural affine change of variables associated with a cube $I$ is denoted 
with the subscript $I$: i.e., for a function $g:\reals^d\to \comps$, 
$$g_I(x) := g\left(\frac{x-c(I)}{\ell(I)}\right).$$

The symbol $C$, often with a subscript, will always represent a constant. The subscript is used to indicate dependence on various parameters. 
The value of $C$ may change, sometimes within the same line.

For  Schwarz functions, the $d$-dimensional Fourier transform is given by the formula
$\widehat{f}(\xi) = \int_{\reals^d} f(x) e^{-i\langle\xi,x\rangle} \dif x$, 
and its inverse is 
$ f(x) = (2\pi)^{-d} \int_{\reals^d} \widehat{f}(\xi) e^{-i\langle x,\xi\rangle} \dif \xi.$
An important property of the Gaussian functions 
\begin{equation}\label{Gaussian}
\phi_{\sigma}: x \mapsto \exp\bigl[-|x/\sigma|^2\bigr],
\end{equation}
is that they satisfy
$\widehat{\phi_{\sigma}}   = (\sigma \sqrt{\pi})^d\phi_{(2/\sigma)}.$

\section{Shift-invariant Gaussian approximation of band-limited functions}
\subsection{Approximation using infinitely many centers}
Let $B \subset\Rd$ be a fixed ball centered at the origin. 
We denote by
\begin{equation}
H_B
\end{equation}
the space of all Schwartz functions whose Fourier transform is supported
in $B$. Let $\phi$ be the $d$-dimensional Gaussian function, dilated
by a fixed  (arbitrary)  dilation $\sigma>0$ (cf. (\ref{Gaussian})). 
Given $h>0$, consider the linear space
$$S_h:=S_h(\phi):=\spam\{\phi(\cdot-\alp):\ \alp\in h\Zd\},$$
closed in the topology, say, of uniform convergence on compact sets.

We consider in this section approximation schemes and approximation errors
for functions in $H_B$ from the space $S_h$. We adopt to this end the
approximation  schemes of \cite{BR}, and show that in our setup these schemes
provide superb approximations to the class $H_B$: the error decays
exponentially fast as the spacing parameter $h$ tends to $0$!

Let us fix now $f\in H_B$, and $h>0$. We denote by $f_\phi$
the function whose Fourier transform is $\hatf/\hatphi$. We note that
$f_\phi$ is in $H_B$,
since $f_{\phi} = f*\eta_{\Phi}$ for a Schwartz function $\eta_{\phi}$ (that
depends only on $\phi$ and $B$) and  $H_B$ is an ideal in the Schwartz 
space.  We then approximate $f$ by $h^d\ful f$,
with
\begin{equation}
\ful f:=
\left(\frac{1}{2\pi}\right)^d\sum_{\alpha \in h\ints^d} f_\phi(\alp) \phi(\cdot - \alpha).
\end{equation}

Our main result in this subsection is the following:
\begin{proposition} \label{NSI}
Let $B= B(0,R)$ be the ball of radius $R$ centered at the origin.
The uniform error in approximating $f\in H_B$ by $h^d\ful f$ as above satisfies,
for $h<\pi/R$,
$$\|f-h^d\ful f\|_{\infty} \le C \|\widehat{f}\|_{L_1} h^d  
e^{\mbox{}-\frac{c}{h^2}}.$$
The constants $C$ and $c$ depend on $R$ and the dilation parameter $\sigma$ 
used in the definition of $\phi$, but are independent of $f$ and $h$.
\end{proposition}
\begin{proof}
Using the fact that $\widehat{f_\phi}=\hatf/\hatphi$, we write $h^d\ful f$ as
$$\int_{\Rd}\hatf(\tet) k_h(\tet,\cdot)\dif\tet,$$
with 
$$k_h(\tet,z):=(2\pi)^{-d}\frac{h^d}{\hatphi(\tet)}\sum_{\alp\in h\Zd}\phi(z-\alp)
e^{i \langle \theta,  \alp \rangle}.$$
Invoking the Poisson summation formula (which obviously is valid for the
Gaussian function), we obtain that
$$k_h(\tet,z)=(2\pi)^{-d}\frac{e^{i \langle z,  \tet \rangle}}
{\hatphi(\tet)}\sum_{\bet\in 2\pi\Zd/h}\hatphi(\tet+\bet)
e^{i \langle z,  \bet \rangle}.$$
When applying the above kernel to $f$, we are allowed to do the integration
term-by-term, with the ($\bet=0$)-term yielding the original function $f$.
Therefore,
$$f(z)-h^d\ful f(z)=\int_{\Rd}\hatf(\tet)k_h'(\tet,z)\dif \tet,$$
with
$$k_h'(\tet,z):=(2\pi)^{-d}\frac{e^{i \langle z,  \tet \rangle}}
{\hatphi(\tet)}\sum_{\bet\in 2\pi\Zd/h\bks \{0\}}\hatphi(\tet+\bet)
e^{i \langle z,  \bet \rangle}.$$
Note that the kernel is integrated only over $\tet\in B$, since
$\supp\hatf\subset B$ by assumption. Thus, we obtain that
$$\|f-h^d\ful f\|_\infty\le (2\pi)^{-d} \norm{\hatf}_1 K_h,$$
with
$$K_h:=\sum_{\bet\in 2\pi\Zd/h\bks
\{0\}}\|\hatphi(\cdot+\bet)/\hatphi\|_{L_\infty(B)}.$$
%
Let $R$ denote the radius of $B$. If $2R< |\bet|$ then, for $\xi\in B$, 
$|\xi+\beta|^2 - |\xi|^2\ge (|\beta| - 2|\xi|)^2$. Consequently
$$K_h\le C_1 \hatphi(a),$$
for $a< \dist_2\bigl(2B, 2\pi\Zd/h\backslash \{0\}\bigr)=2(\pi/h-R).$
\end{proof} 
\subsection{Approximation using finitely many centers}
In this subsection, we modify the approximant of the previous subsection and 
use only a finite number of centers.  This is a necessary step for us,
since our budget of centers is finite. Our approximand is still
a function $f\in H_B$.

Our setup is as follows. Given $f$ and a mesh-scaling parameter $h$,
we will approximate $f$ by $h^d\notful f$, with
\begin{equation}
\notful f:= (2\pi)^{-d}\sum_{\alp\in h\Zd\cap B_h}f_\phi(\alp)\phi(\cdot-\alp),
\end{equation}
with $f_\phi$ and $\phi$ as in the previous subsection, and $B_h$ is
a ball of radius $1/h$. The crux here is the correspondence between
the mesh size $h$, and the radius $1/h$ of the domain of the shifts we
``preserve'': $\notful f$ is obtained from
$\ful f$ by removing from the sum all shifts outside a ball of radius
$1/h$. 
Note that the number of shifts $N:=N(h)$ that are being used for a given
$h$ satisfies
$$N\sim h^{-2d}$$
with constants of equivalence depending on $d$ only. 
At the end, we need to control the error in terms of the parameter $N$.
 For the time being, we still write the error in terms
of the mesh size $h$. 

Once the approximation operator uses the above truncated sum, one cannot
expect the error to decay exponentially fast as in Proposition \ref{NSI}. 
However, the new error, measured in the uniform norm, still decays rapidly:%
\footnote{We could have made the dependence of $C_{f,k}$ below on $f$ more
explicit. However, this is not needed for our subsequent applications.}

\begin{lemma}\label{uniform_err}
Let $k>0$, and $f\in H_B$. Then there exists $C_{f,k}>0$ that does not
depend on $h$ such that for all small enough $h$
%
%
\begin{equation}\label{E:uniform_err}
\norm{f - h^d\notful f}_\infty \le C_{f,k} h^{k} .
\end{equation}
\end{lemma}
\begin{proof}
Thanks to Proposition \ref{NSI}, we only need to show that
$$h^d\norm{\ful f - \notful f}_\infty \le C h^{k}.$$
However, the norm  $\norm{\notful f-\ful f}_{\infty}$ is bounded
above by the sum
$$\sum_{|\alp|>1/h,\alp\in h\Zd}|f_\phi(\alp)|.$$
Since $f_\phi$ decays rapidly at $\infty$, the above sum is $O(h^k)$
for any fixed $k$, and our claim follows.
\end{proof} 


The uniform error bound that we just obtained is not refined enough
for our purposes. We will need better estimates for the error away
from the origin, i.e., outside the ball $B_h$ of radius $1/h$. Indeed, such estimates
are valid, but require a different argument:
\begin{lemma}\label{loc_err}
Let $k>0$, and $f\in H_B$.
Then there is a constant $C_{k}'>0$ (depending on $k$, $d$ and $f$
but independent of $h$), so that the function $\notful f$ from Lemma 
\ref{uniform_err} approximates $f$ with pointwise error:
\begin{equation}\label{E:loc_err}
|(f - h^d\notful f)(x)| \le C_{k}' h^{k} \bigl(1+|x|\bigr)^{-k}.
\end{equation}
\end{lemma}
\begin{proof}
If $|x|\le 2/h$, then 
$$\bigl(1+|x|\bigr)^{-k}\ge (h/3)^k,$$
hence  the requirement here follows from the inequality in Lemma 
\ref{uniform_err} when $k$ there is replaced by $2k$. 

For the case $|x|\ge 2/h$, we may prove that 
$$|(f - h^d\notful f)(x)| \le C_{k}' \bigl(1+|x|\bigr)^{-2k},$$
since
$$h^{k} \bigl(1+|x|\bigr)^{-k}\ge C |x/2|^{-2k}.$$
To this end, we estimate the difference 
$$f(x)-h^d\notful f(x)$$
directly. First, $f$ decays rapidly, by assumption, hence certainly
satisfies the required estimate. As to $h^d\notful f$, we note that,
since $f_\phi$ decays rapidly, the sum
$$h^d\sum_{\alp\in h\Zd}|f_\phi(\alp)|$$
is bounded, and the bound can be made independent of $h$ (the bound is,
essentially, the $L_1$-norm of $f_\phi$). Thus, we can bound
$\notful f(x)$, up to an $h$-independent constant, by
$$\max\{\phi(x-\alp):\ |\alp|\le 1/h\}.$$
Since $|x|\ge 2/h$, $|x-\alp|\ge |x|/2$, hence
$$h^d\notful f(x)\le C_k' \phi(x/2).$$
Thus we are left to show that
$$\phi(x)\le |x|^{-2k},\ \hbox{for $|x|\ge 1/h$},$$
for small enough $h$, which is clearly valid due to the exponential decay of
$\phi$ at $\infty$.
\end{proof}

\subsection{Gaussian approximation of a wavelet system}
We now assume that we have in hand a finite collection $\Psi\subset H_B$,
with $H_B$ as in the previous section. Then, Lemma \ref{loc_err} holds for each 
$f:=\psi\in\Psi$.
Considering $\Psi$ as the set of mother wavelets in a suitable wavelet
system, we need also to develop suitable approximation schemes for shifted
dilations of $\psi$, i.e., we need approximation schemes and
error bounds for functions of the form
$$\psi\bigl((\cdot-c)/\ell\bigr), \quad \psi\in\Psi, \ c\in\Rd, \ \ell>0.$$
However, such schemes are trivial: since we are allowed to use
shifted-dilated
versions of our original Gaussian $\phi$, we may simply use the approximation
$$\psi\bigl((\cdot-c)/\ell)\approx (h^d\notful \psi)((\cdot-c)/\ell\bigr).$$
Note that $\notful \psi$ employs $N\sim h^{-2d}$ centers. Fixing $N$
momentarily, we define a new 
map, $\final N {}$, that is defined
on all dilated shifts of each $\psi\in\Psi$ by
\begin{equation}\label{E:wave_approx}
(\final N
\psi)\bigl((\cdot-c)/\ell)\bigr):=N^{-1/2}(T^\flat_{N^{-1/(2d)}}\psi)\bigl((\cdot-c)/\ell\bigr).
\end{equation}

The error bounds of the previous section apply directly here. We just need
to replace each occurrence of $h$ by $N^{-1/(2d)}$. Thus, we obtain:

\begin{theorem}\label{loc_err3}
Let $\psi\in H_B$ be given and finite.
Let $k>0$, and let $I$ be a cube. 
Then, there exists a constant $C$ independent of $N$ and $I$ such
that, for every $N$ sufficiently large, and for every $I$ as above,
$$|(\psi_I-\final N {\psi_I})(x)|\le C
N^{-k/d}\left(1+\frac{|x-c(I)|}{\ell(I)}\right)^{-2k}.$$
\end{theorem}

\section{Nonlinear Approximation in $L_p$, $1\le p<\infty$}
In the previous section, we derived error estimates for the approximation
of each member of a bandlimited smooth wavelet system by suitably chosen
$N$ shifted-dilated Gaussians. Armed with these error estimates, we finally
tackle in this section our central problem: approximating a general
function by finitely many shifted-dilated Gaussians. Our approach follows
\cite{DR} and is similarly based on approximating the
wavelets in the wavelet expansion of the actual approximand. 
To this end, we choose first any, say orthogonal, wavelet
system whose mother wavelets are all bandlimited Schwartz functions.
We define below MRA systems and wavelets in the exact way that
fits our needs. Let us stress that the actual definitions of wavelet 
systems are far more flexible.

\begin{definition}[Wavelets] In this article a  {\it univariate
wavelet system} is an orthonormal MRA wavelet system  whose
generators are bandlimited Schwartz functions:
a scaling function $\eta_0$ and a (mother) wavelet $\eta_1$, both 
bandlimited Schwartz functions. See \cite[3.2]{Meyer} or \cite{Graf} for a 
possible construction.  Multivariate wavelet systems are tensor products
of a univariate one, hence its wavelets
are indexed by $(I,e)$, an ordered pair 
consisting of a dyadic cube, $I$,  and a gender $e\in \E\in\{0,1\}^d 
\setminus \{\mathbf{0}\}$, corresponding to one of the (non-origin) corners 
of the unit cube $[0,1]^d$:  
$$\psi_{e}(x) = \prod_{j=1}^d \eta_{e_j}(x_j), \quad \psi_{I,e}:=(\psi_e)_I.$$
\end{definition}
Let $\D$ be the collection of all dyadic cubes, viz., 
with $I_0$ the unit cube,
$$\D:=\{2^j(k+I_0):\ j\in\Z,\ k\in\Zd\}.$$
We denote by $\D_j$ the subset of dyadic cubes with common edgelength $2^j$.

The wavelet $\psi_{I,e}$ is an affine change of variable (as in Section 1.4) of 
the mother wavelet $\psi_e=\psi_{I_0,e}$, for some $e\in \E$. 
Since we use more than one mother wavelet (indeed, we use $\#\E = 2^d-1$), we regard 
$\D$ and $\D_j$ as multisets and we suppress dependence on the gender $e$.
Thus, the notation $\psi_I$ stands for the $I$-version of {\it any}
of the mother wavelets, and a summation over $\D$ or over one of its subsets, 
unless otherwise noted, is assumed to take place over $\E$ as well. 
This does not cause any confusion, since in this section
our algorithms and their analysis do not pay attention to the details
of the actual mother wavelet that is employed. 

Our problem is then the following basic one. We are given a smooth function
$f$ (from some smoothness class, see below) and a budget of $N$ centers.
We are then allowed to approximate $f$ by a total of $N$ shifted-dilated Gaussians.
We carry out this approximation by distributing the centers across the wavelet
system: for each $I\in\D$, we allocate $N_I$ centers as ``the $I$-budget'' 
and use these budgeted centers for approximating the term $f_I\psi_I$
in the wavelet expansion 
\begin{equation}\label{wav_exp}
f=\sum_{I\in\D} f_I\psi_I.
\end{equation}
The individual error when approximating $f_I\psi_I$ by $N_I$ Gaussians
was the subject of the previous section. Thus, our analysis here will focus
on the estimation of the cumulative error. But, first and foremost, we need
to devise an algorithm for distributing the  budget across the different
wavelets. We refer to this algorithm as the {\it cost distribution.}

\subsection{Triebel-Lizorkin Cost Distribution}
It is convenient to associate each wavelet with {\bf cost}
$\cd_I>0$ that is not necessarily an integer,
and then to determine $N_I$ from the formula
$$
N_{I} := \begin{cases}\lfloor \cd_I \rfloor,&\quad \lfloor \cd_I \rfloor\ge N_0,\\
	0, &\quad\text{otherwise,}
        \end{cases}
$$
where $N_0$ is a some fixed integer, that depends on the wavelet system
and on  nothing else.

We now discuss the cost  distribution $\cd_I$, which depends on several 
factors. In addition to the volume of the dyadic cube, $|I|$, it depends on 
the wavelet coefficient $f_{I}$,
the smoothness norm of $f$ (defined below), and 
an estimate of the size of a partial reconstruction of $f$. To this end, we 
make the following definitions:
\begin{definition}
Given $s,q>0$,
we define the maximal function $M_{s,q}f$ as
\begin{equation}\label{D:Square_FN}
M_{s,q}f(x)
:=
\left(\sum_{I\in \mathcal{D}} |I|^{-sq/d} |f_{I}|^q\chi_I(x) \right)^{1/q}.
\end{equation}
For a dyadic interval $I$, we define a  partial function by
\begin{equation}\label{D:Partial_Square_FN}
M_{s,q,I}f(x)
:=
\left( \sum_{I\subset I'\in\D}  |I'|^{-sq/d} |f_{I'}|^q \chi_{I'}(x)
\right)^{1/q}.
\end{equation}
Given now $\tau,s,q>0$, we define the Triebel-Lizorkin space
$F_{\tau,q}^s$ via the finiteness of the following
quasi-seminorm:
\begin{equation}\label{D:Trieb_Liz}
|f|_{F_{\tau,q}^{s}} := \|M_{s,q}f\|_{\tau}.
\end{equation}
\end{definition}
We note that for any interval $I$, the partial maximal function $M_{s,q,I}f$
is nonnegative and always $\le $ $M_{s,q}f$. Furthermore, it achieves its 
maximum on the interval $I$, where it is constant. Thus the number
$m_{s,q,I}:= M_{s,q,I}f(x)$, $x\in I$, is well-defined, and
$m_{s,q,I} = \sup_{y\in \reals^d} M_{s,q,I}f(y)\le M_{s,q}f(x),\  x\in I.$
In the definition below, $s$ stands for the smoothness of the function
we approximate, and $p$ for the norm in which we measure the error.
\begin{definition}[Cost Distribution]\label{Cost_Distribution}
Let $s>0$, and $p\ge 1$. Define
$\tau,q$ by $1/\tau := 1/p + s/d$ and $1/q := 1 +s/d$. Let
$f\in F_{\tau,q}^s$, with wavelet expansion (\ref{wav_exp}).
We choose then {\bf the cost of a dyadic cube $I\in\D$} as
\begin{equation}\label{cost_dist}
\cd_I: = |f|_{F_{\tau,q}^{s}}^{-\tau}\, m_{s,q,I}^{\tau-q} \, |f_{I}|^q\, |I|^q N.
\end{equation}
\end{definition}
Let us first verify that the sum of all the costs is our budget $N$:
$$
\sum_{I\in\D} 
  \cd_I 
  =\sum_{I\in\D}|f|_{F_{\tau,q}^{s}}^{-\tau}\, 
  m_{s,q,I}^{\tau-q} \, |f_{I}|^q\, |I|^{1-qs/d} N.
$$
Since $|I|=\int_{\reals^d}\chi_I(x)\dif x$, we can write the right 
hand side as an integral, namely as
$|f|_{F_{\tau,q}^{s}}^{-\tau} N \int_{\reals^d} 
    \sum_I \, m_{s,q,I}^{\tau-q} \, |f_{I}|^q\, |I|^{-qs/d} \chi_I(x) 
  \dif x.$ 
Invoking the fact that, for $x\in I$, $m_{s,q,I}\le M_{s,q}f(x)$  (and that $\tau\ge q$), gives
\begin{eqnarray*}
\sum_I c_{I} 
&\le&
|f|_{F_{\tau,q}^{s}}^{-\tau} N
  \int_{\reals^d} 
    \sum_I \, 
      \bigl(M_{s,q}f(x)\bigr)^{\tau-q} \, |f_{I}|^q\, |I|^{-qs/d} \chi_I(x) 
  \dif x\\
&\le& 
|f|_{F_{\tau,q}^{s}}^{-\tau} N 
  \int_{\reals^d}  
    \bigl(M_{s,q}f(x)\bigr)^{\tau} \, 
  \dif x
=N.
\end{eqnarray*}
%
\subsection{Approximating the Wavelet Expansion}
Once a budget of $N_I$ centers is allocated for the approximation of
the term $f_I\psi_I$ in the wavelet expansion of $f$, we appeal to
Theorem \ref{loc_err3} to conclude that the term can be approximated by
$N_I$ Gaussians with error that is bounded (up to a constant that depends
only on the wavelet system and on the parameter $k$) by $|f_I|R_I$, where
\begin{eqnarray}\label{cost_dist_est}
R_{I}(x)
&:=& 
  C_{k,d}
  \min(1, N_{I}^{-k/d}) 
  \left(1+\frac{\dist(x,I)}{\ell(I)}\right)^{-2k}\nonumber\\
&\le& 
  C_{k,d}'
  \min(1, \cd_I^{-k/d}) 
  \left(1+\frac{\dist(x,I)}{\ell(I)}\right)^{-2k}.
\end{eqnarray}
The following lemma, which is proved in the next subsection, simplifies
the above error:
\begin{lemma}\label{F-S}
Let $1\le p < \infty$, then
$$
\left\|
  \sum_{I\in \D} 
    |f_{I}|\, R_{I}
\right\|_p
\le 
C_{k,d}
\left\|
  \sum_{I\in\D}
    \min(1, \cd_I^{-k/d}) 
    f_{I}\chi_I
\right\|_p.
$$
\end{lemma}
%

We are ready to state and prove our main result concerning the case
$1\le p<\infty$.

\begin{theorem}\label{p_main}
Given $s>0$  and $1\le p <\infty$, there is a constant $C_{p,s,d}$ so that for 
$f\in F_{\tau,q}^s$, with $1/\tau = 1/p + s/d$ and $1/q = 1 +s/d$, there is a
linear combination of $N$ Gaussians
$s_f(x):= \sum_{j=1}^N A_j \exp\bigl[-\left(\frac{x-\xi_j}{\sigma_j}\right)^2\bigr]$ 
so that
$$
\|f - s_f\|_p
\le 
C_{p,s,d} N^{-s/d} |f|_{F_{\tau,q}^s}.
$$
\end{theorem}
\begin{proof}
Using the coefficients of the wavelet expansion (\ref{wav_exp}), we
can express $s_f$ as 
$$s_f:= \sum_{} f_{I}\,T_{N_I}\psi_{I},$$
where each term, 
$\left[T_{N_I}\psi_{I}\right](x) 
= 
\sum_{j=1}^{N_I} 
  a_{I,j} \exp \left[-\left(\frac{x-c_{I,j}}{\ell(I)}\right)^2\right],$ 
defined in (\ref{E:wave_approx}), is composed of $N_I$ Gaussians by the 
construction preceding Theorem \ref{loc_err3} (note that the notation
$\cd_I$ stands for the $I$-cost, and is very different from the notation 
$c_{I,j}$ above). By the enumeration at the end 
of Section 3.1 ($\sum_{I\in\D} N_I \le N$), we know that no more than $N$ 
Gaussians are used.

From Lemma \ref{F-S} we have the error estimate 
$$
\|f - s_f\|_p
\le 
C_{k,d}
\left\|
  \sum_{I\in\D} 
    \min(1, \cd_I^{-k/d}) 
    |f_{I}|\chi_I
\right\|_p.
$$
As long as $k$ (which is arbitrary) is greater than $s$, 
we can estimate the error as the $L_p$ norm of a series:
$$\|f - s_f\|_p 
\le 
C_{k,d}
\left\|
  \sum_{I\in\D} 
    E_I
\right\|_p,
$$ 
where 
$E_I(x):= \cd_I^{-s/d}\,|f_{I}|\,\chi_I(x).$
We now focus on estimating this series, pointwise.

By applying the definition of $\cd_I$, we obtain 
(after some elementary manipulation of exponents),
$
\cd_I^{-s/d} |f_I|  
=  
|f|_{F_{\tau,q}^s}^{\tau s/d}
m_{q,s,I}^{\tau/p-q} 
|f_I|^{q} \,   |I|^{-qs/d}  N^{-s/d}.
$
We recall that the $I^{\text{th}}$ partial square-like 
function is constant on the cube $I$, where it equals 
$m_{q,s,I}$. This implies that 
$\chi_I(x) m_{q,s,I} = \chi_I(x) M_{q,s,I}(x)$, 
which shows that each term is
$$
E_I(x) 
=
N^{-s/d} 
|f|_{F_{\tau,q}^s}^{\tau s/d} 
M_{q,s,I}(x)^{\tau/p-q} 
|f_I|^{q}\, 
|I|^{-qs/d} 
\chi_I(x).
$$
The series becomes much more manageable by making some simple substitutions.
Writing the basic summand of the maximal function as 
$z_I:= |f_I|^{q} |I|^{-qs/d} \chi_I(x)$, 
the partial sum of these basic summands, 
$Z_I:= \sum_{I\subset I'} z_{I'}$, 
is observed to be the $q^{\text{th}}$ power of the partial maximal function
$Z_I = \bigl(M_{s,q,I}(x)\bigr)^q$, 
while the full sum of these, 
$Z:=\sum_{I\in \D}  z_I$, 
is simply the $q{\text{th}}$ power of the (full) maximal function
$Z = \bigl(M_{s,q}(x)\bigr)^q$.
It is a simple observation that the full series under consideration 
now has the compact form
$$
\sum_{I\in\D} E_I(x) 
= 
N^{-s/d} |f|_{F_{\tau,q}^s}^{\tau s/d} 
\sum_{I\in\D}
  z_I Z_I^{\frac{\tau}{pq}-1}.
$$
It follows from the inequality 
$\sum_{I\in\D} z_I Z_I^{\epsilon-1}\le C_{\epsilon} Z^{\epsilon}$, 
valid for nonnegative sequences $(z_I)_{I\in \D}$ and  
$0<\epsilon$ with constant $C_\epsilon<\infty$ 
(this is \cite{DR}[Lemma 6.3]), that 
$$
\sum_{I\in\D} E_I(x)
\le 
C_{\frac{\tau}{pq}}
N^{-s/d} |f|_{F_{\tau,q}^s}^{\tau s/d} 
\left(
  \bigl(M_{s,q}(x)\bigr)^q
\right)^{\frac{\tau}{pq}} 
= 
C_{p,s,d} 
N^{-s/d} |f|_{F_{\tau,q}^s}^{\tau s/d} 
\bigl(M_{s,q}(x)\bigr)^{\tau/p}.
$$
Taking the $L_p$ norm controls the error:
\begin{eqnarray*}
\left(
  \int_{\reals^d}
    \left(
      \sum_{I\in\D} 
        E_I(x)
    \right)^p
  \dif x
\right)^{1/p}
&\le& 
C_{p,s,d} 
N^{-s/d} |f|_{F_{\tau,q}^s}^{\tau s/d} 
\left(
  \int_{\reals^d} 
    \bigl(M_{s,q}(x)\bigr)^{\tau} 
  \dif x
\right)^{1/p}\\ 
&=& 
C_{p,s,d} 
N^{-s/d} |f|_{F_{\tau,q}^s}^{\tau s/d +\tau/p} 
=  
C_{p,s,d} 
N^{-s/d} |f|_{F_{\tau,q}^s}
\end{eqnarray*}
since $\tau s/d +\tau/p =1$.
\end{proof}
%

\subsection{On Lemma \ref{F-S}}
The vector-valued maximal inequality of Fefferman and Stein, \cite[Theorem 
1]{FeSt},    controls the $L_r(\ell_s)$ norm of the sequence of functions 
$(MF_j)_j$ by the $L_r(\ell_s)$ norm of $(F_j)_j$, provided $1<r,s<\infty$
(the operator $M$ is the usual Hardy--Littlewood maximal operator
$MF(x):= \sup_{x\in [a,b]^d} \frac{1}{(b-a)^d} \int_{[a,b]^d}|F(y)|\dif y$ ):
$$
\left\|\left(
  \sum |MF_j(x)|^s
\right)^{1/s}\right\|_r 
\le 
C_{r,s} 
\left\|\left(
  \sum |F_j(x)|^s
\right)^{1/s}\right\|_r
$$
In the lemma we make use of a minor generalization of this for the 
modified maximal operator $M_{\tau}$, defined for $0<\tau<\infty$ by
$$
M_{\tau}f(x)
:= 
\sup_{x\in [a,b]^d} 
\left(
  \frac{1}{(b-a)^d} 
    \int_{[a,b]^d}
      |f(y)|^{\tau}
    \dif y 
\right)^{\frac{1}{\tau}}.$$
It is not difficult to show that for $\tau<p,q<\infty$,
\begin{equation}\label{H-L}
\left\|\left(
  \sum 
    |M_{\tau}f_j(x)|^q
\right)^{1/q} \right\|_p 
\le 
K 
\left\|\left(
  \sum 
    |f_j(x)|^q
\right)^{1/q}\right\|_p.
\end{equation}
Indeed, this follows by a direct application of the Fefferman--Stein inequality with  $s=\tfrac{q}{\tau}$, $r=\tfrac{p}{\tau}$ (both greater than one), 
$F_j = f_j^{\tau}$ and $K=C_{r,s}^{1/\tau}$,
because the modified maximal operator is related to the Hardy-Littlewood maximal operator by $[MF_j]^s = [M_{\tau}f_j]^q$ and the $r$ and $p$ norms are related by
 $\|g\|_r = \|g^{1/\tau}\|_p^{\tau}.$

\begin{proof}[Proof of Lemma \ref{F-S}]
From (\ref{cost_dist_est}), it follows that
$$R_{I}\le C_{k,d}
\min(1, \cd_I^{-k/d}) \left(1+\frac{\dist(x,I)}{\ell(I)}\right)^{-2k}.$$ 
Observe that there is a constant, $C_d$, depending only on $d$ so that
$$ 
\left(
  1+\frac{\dist(x,I)}{\ell(I)}
\right)^{-d}
\le 
M(\chi_I)(x).
$$
We can assume $k>d/2$ without loss of generality. For $\tau$ between $d/2k$ and $1$ we have
$$
\left
  (1+\frac{\dist(x,I)}{\ell(I)}
\right)^{-2k}
\le 
\left
  (1+\frac{\dist(x,I)}{\ell(I)}
\right)^{-d/\tau}
\le 
C_d M_{\tau}(\chi_I)(x),
$$
since $M_{\tau}(\chi_I) = M(\chi_I)^{1/\tau}.$
It follows from the modified Fefferman-Stein inequality (\ref{H-L}), that
\begin{eqnarray*}
\left\|\sum_{I\in \mathcal{D}} |f_{I}|\, R_I\right\|_p 
&\le&
C_{k,d}
\left\|
  \sum_{I\in \mathcal{D}} 
    |f_{I}|\, 
    \min(1, \cd_I^{-k/d}) 
    M_{\tau}(\chi_I)
\right\|_p\\
&\le&
C_{k,d}
\left\|
  \min(1, \cd_I^{-k/d}) 
  |f_{I}| \chi_I 
\right\|_p.
\end{eqnarray*}
\end{proof}
%
\section{Nonlinear Approximation in $L_{\infty}$}
Although the basic strategy for nonlinear RBF approximation in $L_{\infty}$ is,
at heart, the same as in $L_p$, there are some complications that require us to
 give it a slightly different treatment. The fundamental difference is that 
the Hardy-Littlewood maximal inequality 
(and, hence, its vector valued analogue, the Fefferman-Stein inequality, used in the previous section) does not hold for $p=\infty$. 
For this reason, we choose to work with family of smoothness spaces 
that do not require us to explicitly work with maximal operators. Smoothness
is measured using a Besov norm, and we use a Besov space based cost
distribution to determine how to distribute the budget.\footnote{The Besov
space approach is valid for the case $p<\infty$ that was analysed in the
previous section, too. However, the Triebel-Lizorkin space $F^s_{\tau,q}$ is
slightly larger than the Besov space of the same parameters.}

\begin{definition}
For $\tau = d/s\in (0,\infty)$ and $q\in (0,\infty)$, the Besov space $B_{\tau,q}^{s}$
is the space of $L_{\tau}$ functions for which the (quasi-)seminorm
$|f|_{B_{\tau,q}^{s}}$ is finite, where
$$
|f|_{B_{\tau,q}^{s}} 
:= 
\left\|
  k 
  \mapsto 
  \left(
    \sum_{I\in \mathcal{D}_k}
         |f_{I}|^{\tau}
  \right)^{1/\tau}
\right\|_{\ell_q(\ints)}.
$$
Here, the coefficients $(f_I)$ are as in (\ref{wav_exp}).
\end{definition}
Note that for  
$q\le \tau\le1$, $f\in B_{\tau,q}^{s}$ implies that the  wavelet coefficients $f_{I}$ are absolutely summable. Since the wavelets $\psi_{I}$ are uniformly bounded, this means that the wavelet expansion (\ref{wav_exp}) is absolutely convergent for $s\ge d$ and $f\in B_{\tau,q}^s$ (meaning that the main issue for $L_{\infty}$ approximation is resolved in this case). For $1<\tau<\infty$ and $q\le1$, we also have unconditional convergence of the wavelet expansion, since $$\sum_{k\in \ints} \sum_{I\in \D_k} |f_I| |\psi_I(x)| \le \sum_{k\in \ints} \left(\sum_{I\in \D_k} |f_I|^\tau\right)^{1/\tau} \left(\sum_{I\in\D_k} |\psi_I(x)|^{\tau'} \right)^{1/\tau'}.$$
Because $|\psi_I(x)|^{\tau'} \le C\bigl(1+|x-c(I)|/\ell(I)\bigr)^{-(d+1)}$, the second factor is bounded with a constant depending on $d$ and totally independent of $k$ and $x$. Thus, the right hand
side is controlled by $|f|_{B_{\tau,q}^{s}}$. This is a reflection of the fact that 
$B_{\tau,q}^s$ is embedded in $L_{\infty}$ for $\tau = d/s$ and $q\le \min(\tau,1)$.  Although $L_{\infty}$ has no unconditional basis, the Besov space does; the wavelet expansion
(\ref{wav_exp}) converges unconditionally in these cases.

\subsection{\bf  Besov Cost Distribution:}
The approach we take for treating $L_{\infty}$ error is to alter the strategy for budgeting slightly.
As before, for each wavelet $\psi_{I}$, we create an approximant 
$\gapp{N_{I}}\psi_{I}$ using a portion $N_{I}$ of the total budget $N$, 
but the precise distribution of this budget follows different rules. 
We rely again on a cost distribution. 
In this case, it is:
\begin{equation}\label{B_cost_dist}
\cd_{I} = N |f|_{B_{\tau,q}^s}^{-q} A_j^{q-\tau} |f_{I}|^{\tau}
=N \left(\frac{A_j}{|f|_{B_{\tau,q}^s}}\right)^q
\left(\frac{|f_{I}|}{A_j}\right)^\tau.
\end{equation}

The indices $\tau$ and $q$ are determined by $1/\tau = s/d$ and $1/q = 1 +s/d$.
The quantity $A_j$ is a sort of ``energy'' of $f$ at the dyadic level $j$:
$$A_j:= \left( \sum_{I\in \D_j} |f_{I}|^{\tau}\right)^{1/\tau}.$$
We  do not invest in the wavelet corresponding to 
$I$ if  $\lfloor \cd_{I}\rfloor \le N_0$ (the constant from
Lemma \ref{loc_err}).  Thus, we set 
\begin{equation}\label{first_inf_budget}
N_{I} 
= 
\begin{cases}
  \lfloor \cd_{I}\rfloor 
  = 
  \left\lfloor 
    N |f|_{B_{\tau,q}^s}^{-q} A_j^{q-\tau} |f_{I}|^{\tau}
  \right\rfloor,\quad
  &
  \lfloor \cd_{I}\rfloor \ge N_0\\
  0&\text{otherwise.}
\end{cases}
\end{equation}
With this choice at most $N$ Gaussians are used:
\begin{eqnarray*}
\sum_{I\in \D} N_{I} &\le& 
\sum_{I\in \D} \cd_{I}
\le
N |f|_{B_{\tau,q}^s}^{-q}  
\sum_{j=0}^{\infty} 
  A_j^{q-\tau}
  \sum_{I\in \D_j}
    |f_{I}|^{\tau}\\
&=& 
N |f|_{B_{\tau,q}^s}^{-q}  
\sum_{j=0}^{\infty} 
  A_j^{q-\tau} A_j^{\tau} = N.\\
\end{eqnarray*}

\subsection{Approximating the Wavelet Expansion:}
The following lemma is
a rough analogue to Lemma \ref{F-S}, where we show that the investment of centers made
in one ``energy level'' gives a suitable error.
\begin{lemma}\label{energy}
For a finitely supported sequence of coefficients $(a_{I})_{I\in \D_j}$, we have the estimate, which holds for $2k>d$:
$$
\left\|
  \sum_{I\in\mathcal{D}_j}
    a_{I} \psi_{I} 
  -
   \sum_{I\in\mathcal{D}_j}
     a_{I} [\gapp{N_{I}}\psi_{I}]
\right\|_{\infty} 
\le 
C_{k,d} \sup_{I\in\D_j}
\left|a_{I}  N_{I}^{-k/d}\right|.
$$
\end{lemma}
\begin{proof}
We treat the estimate by considering the error termwise.  
Theorem \ref{loc_err3}  gives the pointwise bound
$$
  \sum_{I\in\mathcal{D}_j}
    |a_{I}| 
    \left|
      \psi_{I}(x) - \left[\gapp{N_{I}}\psi_{I}\right](x)
    \right|
\le 
C_{k,d}
\sum_{I\in\mathcal{D}_j}  
  |a_{I}| N_{I}^{-k/d}  
  \left(1+\frac{\dist(x,I)}{\ell(I)}\right)^{-2k}.
$$
By applying H{\" o}lder's inequality, the lemma follows, since for $2k>d$, the series $\sum_{I\in \D_j} 
\left(1+\frac{\dist(x,I)}{\ell(I)}\right)^{-2k}$ is bounded by a finite  
constant $C_{k,d}$ that is independent of both $j$ and $x$. 
\end{proof}
We are now in a position to prove our main result for $L_{\infty}$ 
approximation.
\begin{theorem}\label{inf_first}
Given $s>0$, there is a constant $C_{s,d}$ so that for 
$f\in B_{\tau,q}^s$, with $1/\tau =  s/d$ and $1/q = 1 +s/d$, there is a
linear combination of $N$ Gaussians
$s_f(x):= \sum_{j=1}^N A_j \exp\left[-(\frac{x-\xi_j}{\sigma_j})^2\right]$ 
so that
$$
\|f - s_f\|_{\infty}
\le 
C_{s,d} N^{-s/d} |f|_{B_{\tau,q}^s}.
$$
\end{theorem}
\begin{proof}
Using the budget (\ref{first_inf_budget}),
the approximant is
$$s_f:= \sum_{I\in \D_k} f_{I}\,T_{N_{I}}\psi_{I},$$
where each term, 
$\left[T_{N_{I}}\psi_{I}\right]$ 
is composed of $N_{I}$ Gaussians as in Theorem \ref{loc_err3}.

We estimate $\|f - s_f\|_{\infty},$ recalling the unconditional convergence
of the wavelet expansion for functions coming from the Besov space for this 
choice of $\tau$ and $q$.
\begin{eqnarray*}
\|f - s_f\|_{\infty}
&\le& 
\sum_{j=-\infty}^{\infty} 
\left\|
  \sum_{I\in \D_j} 
    f_{I} (\psi_{I} - \gapp{N_{I}}\psi_{I})
\right\|_{\infty}\\
&\le&
C(k,d) 
\sum_{j=-\infty}^{\infty} 
  \left\|
    I\mapsto  N_{I}^{-k/d}|f_{I}|
  \right\|_{\ell_{\infty}(\D_j)}\\
&\le&
C(k,d) 
\sum_{j=-\infty}^{\infty} 
  \left\|
    I\mapsto  \cd_{I}^{-s/d}|f_{I}|
  \right\|_{\ell_{\infty}(\D_j)}.\\ 
\end{eqnarray*}
The first inequality is simply the triangle inequality, since the sums considered
are all finite, while the second is Lemma \ref{energy}. The final inequality holds
for $s<k$, because $\cd_{I}\le 1+ N_{I}\le C N_{I}$.

By invoking the definition of $\cd_{I}$, and by manipulating exponents
(specifically, using the facts that $\tau s/d = 1$ and that $s/d+1 = 1/q$) 
we arrive at
\begin{eqnarray*}
\|f - s_f\|_{\infty}
&\le&
C(k,d) N^{-s/d} |f|_{B_{\tau,q}^s}^{qs/d} 
\sum_{j=-\infty}^{\infty} A_j^{(\tau-q)s/d}
  \left\|
    I\mapsto  |f_{I}|^{1-\tau s/d}
  \right\|_{\ell_{\infty}(\D_j)}\\
 &=& C(k,d) N^{-s/d} |f|_{B_{\tau,q}^s}^{qs/d} 
\sum_{j=-\infty}^{\infty} A_j^{(\tau-q)s/d} \\
&=& 
C(k,d) N^{-s/d} |f|_{B_{\tau,q}^s}^{qs/d} 
\sum_{j=-\infty}^{\infty} A_j^{q} = 
C(k,d) N^{-s/d} |f|_{B_{\tau,q}^s}.
\end{eqnarray*}
\end{proof}

\bibliographystyle{siam}
\bibliography{HR_rev}

\end{document}